\documentclass[reqno]{amsart}
\usepackage{latexsym,amssymb,bbm,amsmath,mathabx,amsthm}
\usepackage[pdftex]{graphicx}
\usepackage[scr=boondox]{mathalfa}
\usepackage[mathscr]{euscript}
\usepackage{rotating}
\usepackage{color}
\usepackage[color,emtex,matrix,arrow]{xy}
\usepackage{xypic}
\usepackage{comment}
\usepackage{bbold}
\usepackage{calligra}
\usepackage[T1]{fontenc}
\usepackage{calrsfs}
\usepackage{tikz}

\newtheorem{thm}{\sc Theorem}[section]
\newtheorem{cor}[thm]{\sc Corollary}
\newtheorem{ex}[thm]{\sc Example}
\newtheorem{exs}[thm]{\sc Examples}
\newtheorem{notacio}[thm]{\sc Notation}

\newtheorem{prop}[thm]{\sc Proposition}
\newtheorem{prop-defn}[thm]{\sc Proposition-Definition}
\newtheorem{defn}[thm]{\sc Definition}
\newtheorem{rem}[thm]{\sc Remark}
\numberwithin{equation}{section}

\newcommand{\Cc}{\mathcal{C}}
\newcommand{\Dd}{\mathcal{D}}

\newcommand{\Gg}{\mathcal{G}}

\newcommand{\Rr}{\mathcal{R}}

\newcommand{\Rsc}{\mathscr{R}}
\newcommand{\Ssc}{\mathscr{S}}

\newcommand{\BB}{\mathbb{B}}
\newcommand{\CC}{\mathbb{C}}
\newcommand{\DD}{\mathbb{D}}

\newcommand{\FF}{\mathbb{F}}
\newcommand{\GG}{\mathbb{G}}

\newcommand{\RR}{\mathbb{R}}

\newcommand{\ZZ}{\mathbb{Z}}

\newcommand{\Asf}{\mathsf{A}}
\newcommand{\Bsf}{\mathsf{B}}

\newcommand{\Gsf}{\mathsf{G}}
\newcommand{\Hsf}{\mathsf{H}}

\newcommand{\Rsf}{\mathsf{R}}

\newcommand{\Zsf}{\mathsf{Z}}

\newcommand{\asf}{\mathsf{a}}

\newcommand{\ksf}{\mathsf{k}}

\newcommand{\zsf}{\mathsf{z}}

\def\1{^{-1}}
\def\cref#1#2#3{\left(#2\right.\left|\ #3\right)_{#1}}

\let\lim\varprojlim

\begin{document}

\title{On the (algebraic) notion of 2-ring}
\author{Josep Elgueta}
\date{}
\address{Departament de Matem\`atiques, Universitat Polit\`ecnica de Catalunya}
\email{josep.elgueta@upc.edu}

\maketitle

\begin{abstract}
By a 2-ring we mean a groupoid with a structure analogous to that of a ring, up to coherent isomorphisms. Two different notions of 2-ring appear in the literature: the notion of {\em Ann-category}, due to Quang \cite{Quang-1987}, and the notion of {\em categorical ring}, due to Jibladze and Pirashvili \cite{Jibladze-Pirashvili-2007}. The underlying data are the same in both cases, but the required axioms differ. In \cite{Quang-Hanh-Thuy-2008}, it is claimed that both definitions are equivalent, modulo an additional axiom that must be added to the Jibladze--Pirashvili definition. In this note, we clarify the relationship between these notions by showing how this additional axiom arises and why it must be imposed. Essential to this analysis is the equivalent description of a symmetric monoidal category given in \cite{Elgueta-2025}.
\end{abstract}

\section{Introduction}

By a 2-ring we mean a {\em groupoid} equipped with a structure analogous to that of a ring but with all axioms holding only up to natural isomorphisms, which must satisfy their own axioms. In particular, the underlying abelian group is replaced by a groupoid analog of an abelian group, usually a symmetric 2-group. 

Two concrete notions of 2-ring have been introduced in the literature: the notion of {\em Ann-category}, due to Quang \cite{Quang-1987}, and the notion of {\em categorical ring}, due to Jibladze and Pirashvili \cite{Jibladze-Pirashvili-2007}. The underlying data are the same in both cases but the required axioms differ. In \cite{Quang-Hanh-Thuy-2008}, it is shown that every Ann-category is a categorical ring, and that the converse holds if a suitable axiom is added to the Jibladze--Pirashvili definition. A few years later, Quang \cite{Quang-2013} provided a specific example of a categorical ring that is not an Ann-category, thereby demonstrating that the new axiom is indeed necessary. Despite this, some more recent papers still refer to the {\em apparently} missing axiom in the Jibladze--Pirashvili definition (e.g. \cite{Aldrovandi-2017}, p.~931), or cite the paper by Jibladze and Pirashvili as if there were no gap in their notion of categorical ring (e.g. \cite{Aldrovandi-Lester-2023}, p.~799).

The purpose of this note is to clarify why the Jibladze--Pirashvili definition is wrong by explaining how the additional axiom arises and why it must be imposed for the two definitions to be equivalent. Roughly, the key point is that symmetric 2-groups are equivalent to what we call AC 2-groups; see \cite{Elgueta-2025}. These differ from symmetric 2-groups in that the separate associator and commutator isomorphisms, satisfying the usual coherence axioms, are replaced by unified associo-commutators satisfying their own coherence axioms. Both structures are fully equivalent, in the sense that the corresponding 2-categories are isomorphic and have the same morphisms and 2-morphisms. However, while the description of morphisms in the 2-category of symmetric 2-groups can be simplified, no such simplification is possible in the 2-category of AC 2-groups. As a consequence, in the definition of an Ann-category some of the data ---and all axioms involving them--- become redundant, whereas this is not the case in the definition of a categorical ring, as implicitly assumed by Jibladze and Pirashvili. 

The outline of the paper is as follows. In Section~2 we recall the notions of (symmetric) monoidal category, (symmetric) monoidal functor and monoidal natural transformation, as well as the alternative equivalent notions of AC category, AC functor and AC natural transformation as defined in \cite{Elgueta-2025}. We also discuss the equivalent notions of symmetric 2-group and AC 2-group, and the corresponding notions of morphism and 2-morphism. In particular, the question of the data and axioms that describe the morphisms and 2-morphisms in each 2-category is discussed in detail. In Section~3 we recall the notions of 2-ring due to Quang and to Jibladze--Pirashvili, and we introduce the third notion of {\em AC 2-ring}, which is completely equivalent to Quang's notion, and from which Jibladze--Pirashvili notion is a non-trivial generalization.

We assume that the reader is already familiar with the (weak) notions of 2-category, 2-functor and 2-natural transformation (see \cite{Benabou-1967}, \cite{Borceux-1994-I}, \cite{Kelly-Street-1974}, \cite{MacLane-1998} or the more recent and comprehensive book \cite{Johnson-Yau-2021}). By contrast, the definitions of (symmetric) monoidal category, (symmetric) monoidal functor, and monoidal natural transformation are recalled because of their key role in what follows.

\section{Preliminaries}
\label{repas_cat_monoidals}

\subsection{Symmetric monoidal categories and AC categories}

The standard definition of symmetric monoidal category, due to MacLane \cite{MacLane-1963}, is as follows.

\begin{defn}\label{categoria_monoidal_simetrica}
A {\em symmetric monoidal category} is a 7-tuple $\CC=(\Cc,\oplus,0,a,c,l,r)$ consisting of:
\begin{itemize}
\item a category $\Cc$;
\item a functor $\oplus:\Cc\times\Cc\to\Cc$ (the {\em sum functor});
\item a distinguished object $0$ (the {\em identity object});
\item a family of natural isomorphisms $a_{x,y,z}:x\oplus(y\oplus z)\to(x\oplus y)\oplus z$ (the {\em associators});
\item a family of natural isomorphisms $c_{x,y}:x\oplus y\to y\oplus x$ (the {\em commutators});
\item two families of natural isomorphisms $l_x:0\oplus x\to x$ and $r_x:x\oplus 0\to x$ (the {\em left} and {\em right unitors}).
\end{itemize}
These data must satisfy the following axioms:
\begin{itemize}
\item[(SC1)] {\em (Pentagon axiom)} for every objects $x,y,z,t$ the following diagram commutes:
\[
\xymatrix{
x\oplus(y\oplus(z\oplus t))\ar[rr]^-{id\oplus a_{y,z,t}}\ar[dd]_{a_{x,y,z\oplus t}} && x\oplus((y\oplus z)\oplus t)\ar[d]^{a_{x,y\oplus z,t}} 
\\ && (x\oplus(y\oplus z)\oplus t\ar[d]^{a_{x,y,z}\oplus id}
\\ (x\oplus y)\oplus (z\oplus t)\ar[rr]_{a_{x\oplus y,z,t}} && ((x\oplus y)\oplus z)\oplus t
}
\]
\item[(SC2)] {\em (Triangle axiom)} for every objects $x,y$ the following  diagram commutes:
\[
\xymatrix{ x\oplus(0\oplus y)\ar[rr]^-{a_{x,0,y}}\ar[rd]_{id\oplus l_y} && (x\oplus 0)\oplus y\ar[ld]^{r_x\oplus id}
\\ & x\oplus y &
}
\]
\item[(SC3)] {\em (Hexagon axiom)}  for every objects $x,y,z$ the following diagram commutes:
\[
\xymatrix@R=1pc{
& (x\oplus y)\oplus z\ar[r]^-{c_{x\oplus y,z}} & z\oplus(x\oplus y)\ar[rd]^{a_{z,x,y}} & 
\\
x\oplus(y\oplus z)\ar[ru]^{a_{x,y,z}}\ar[rd]_{id\oplus c_{y,z}} & & & (z\oplus x)\oplus y
\\
& x\oplus(z\oplus y)\ar[r]_-{a_{x,z,y}} & (x\oplus z)\oplus y\ar[ru]_{c_{x,z}\oplus id} & 
} 
\]
\item[(SC4)] {\em (Symmetry condition)} for every objects $x,y$ in $\Cc$ we have $c_{y,x}\,c_{x,y}=id_{x\oplus y}$.
\end{itemize}
\end{defn}

\begin{rem}{\rm
The data $(\Cc,\oplus,0,a,l,r)$ satisfying axioms (SC1)-(SC2) is called a {\em monoidal category}. In any (symmetric) monoidal category $\CC$ the isomorphisms $l_0,r_0$ coincide (see \cite{Kelly-1964}). This isomorphism is denoted by $d:0\oplus 0\to 0$, and called the {\em basic unitor} of $\CC$. 
}
\end{rem}

As argued below, there is a notion equivalent to a symmetric monoidal category in which, instead of separate associator and commutator isomorphisms satisfying the above coherence axioms, we have "associo-commutator" isomorphisms satisfying
appropriate coherence laws. These structures are called {\em AC categories} in \cite{Elgueta-2025}.

\begin{defn} \label{AC-categoria}
An {\em AC category} is a 6-tuple $\CC=(\Cc,\oplus,0,b,l,r)$, where the data $\Cc,\oplus,0,l,r$ are as in Definition~\ref{categoria_monoidal_simetrica}, and $b$ is a familly of natural isomorphisms 
\[
b(x,y,z,t):(x\oplus y)\oplus(z\oplus t)\to(x\oplus z)\oplus(y\oplus t)
\]
(the {\em associo-commutators}). These data must satisfy the following axioms:
\begin{itemize}
\item[{\rm (AC1)}] {\em ($4\times 4$ axiom)} for every objects $x,y,z,t,x',y',z',t'$ in $\Cc$ the following diagram commutes:
{\small
\[
\xymatrix@C=4pc{
[(x\oplus y)\oplus(z\oplus t)]\oplus[(x'\oplus y')\oplus(z'\oplus t')]\ar[rr]^-{b(x\oplus y,z\oplus t,x'\oplus y',z'\oplus t')}\ar[d]_-{b(x,y,z,t)\oplus b(x',y',z',t')}
&&
[(x\oplus y)\oplus(x'\oplus y')]\oplus[(z\oplus t)\oplus(z'\oplus t')]\ar[d]^-{b(x,y,x',y')\oplus b(z,t,z',t')} 
\\
[(x\oplus z)\oplus(y\oplus t)]\oplus[(x'\oplus z')\oplus(y'\oplus t')]\ar[d]_-{b(x\oplus z,y\oplus t,x'\oplus z',y'\oplus t')} 
&&
[(x\oplus x')\oplus(y\oplus y')]\oplus[(z\oplus z')\oplus(t\oplus t')]\ar[d]^{b(x\oplus x',y\oplus y',z\oplus z',t\oplus t')}
\\
[(x\oplus z)\oplus(x'\oplus z')]\oplus[(y\oplus t)\oplus(y'\oplus t')]\ar[rr]_-{b(x,z,x',z')\oplus b(y,t,y',t')} && [(x\oplus x')\oplus(z\oplus z')]\oplus[(y\oplus y')\oplus(t\oplus t')]
}
\] }

\item[{\rm (AC2)}] {\em (unital axioms)} for every objects $x,y$ in $\Cc$ the following diagrams commute:
{\small
\begin{equation}\label{eq00}
\xymatrix{
(x\oplus 0)\oplus(y\oplus 0)\ar[rr]^-{b(x,0,y,0)}\ar[d]_-{r_x\oplus r_y} && (x\oplus y)\oplus(0\oplus 0)\ar[d]^-{id\oplus l_0}
\\
x\oplus y && (x\oplus y)\oplus 0\ar[ll]^-{r_{x\oplus y}} }
\quad 
\xymatrix{
(0\oplus x)\oplus(0\oplus y)\ar[rr]^-{b(0,x,0,y)}\ar[d]_-{l_x\oplus l_y} && (0\oplus 0)\oplus(x\oplus y)\ar[d]^-{r_0\oplus id}
\\
x\oplus y && 0\oplus(x\oplus y)\ar[ll]^-{l_{x\oplus y}} }
\end{equation}
\begin{equation}\label{eq0}
\xymatrix{
(x\oplus y)\oplus(0\oplus 0)\ar[rr]^-{b(x,y,0,0)}\ar[d]_-{id\oplus l_0} && (x\oplus 0)\oplus(y\oplus 0)\ar[d]^-{r_x\cdot r_y}
\\
(x\oplus y)\oplus 0\ar[rr]_-{r_{x\oplus y}} && x\oplus y }
\quad 
\xymatrix{
(0\oplus 0)\oplus(x\oplus y)\ar[rr]^-{b(0,0,x,y)}\ar[d]_-{r_0\oplus id} && (0\oplus x)\oplus(0\oplus y)\ar[d]^-{l_x\cdot l_y} 
\\
0\oplus(x\oplus y)\ar[rr]_-{l_{x\oplus y}} && x\oplus y }
\end{equation} }
\item[{\rm (AC3)}] {\em (normalization axiom)} for every objects $x,y$ in $\Cc$ we have
\begin{equation}\label{cond_simetria_b}
b(x,0,0,y)=id_{(x\oplus 0)\oplus(0\oplus y)}\,;
\end{equation}
\end{itemize}
\end{defn}

Each type of structure has its own notion of morphism. In both cases, the data is the same, but they differ on the required axioms. 

\begin{defn}\label{functor_monoidal_simetric}
Given two symmetric monoidal categories $\CC,\CC'$, a {\em symmetric monoidal functor} between them is a triple $\FF=(F,F_\oplus,F_0)$ consisting of:
\begin{itemize}
\item a functor $F:\Cc\to\Cc'$;
\item a family of natural isomorphism $F_\oplus(x,y):Fx\oplus' Fy\to F(x\oplus y)$ (the {\em monoidality isomorphisms});
\item an isomorphism $F_0:0'\to F0$ (the {\em zero isomorphism}). 
\end{itemize}
These data must satisfy the following axioms:
\begin{itemize}
\item[(SF1)] for every objects $x,y,z$ the following diagram commutes:
\[
\xymatrix@C=3pc{
Fx\oplus'(Fy\oplus' Fz)\ar[d]_{id\oplus' F_\oplus(y,z)}\ar[r]^-{a'} & (Fx\oplus' Fy)\oplus' Fz\ar[r]^-{F_\oplus(x,y)\oplus' id} & F(x\oplus y)\oplus' fz\ar[d]^{F_\oplus(x\oplus y,z)}
\\
Fx\oplus' F(y\oplus z)\ar[r]_-{F_\oplus(x,y\oplus z)} & F(x\oplus(y\oplus z))\ar[r]_-{Fa} & F((x\oplus y)\oplus z)
}
\]
\item[(SF2)] for every objects $x,y$ the following diagram commutes:
\[
\xymatrix{
Fx\oplus' Fy\ar[d]_{F_\oplus(x,y)}\ar[rr]^{c'_{Fx,Fy}} && Fy\oplus' Fx\ar[d]^{F_\oplus(y,x)}
\\
F(x\oplus y)\ar[rr]_{F(c_{x,y})} && F(y\oplus x)
}
\]

\item[(SF3)] for every object $x$ the following diagrams commute:
\[
\xymatrix{ Fx\oplus' 0'\ar[r]^-{id\oplus' F_0}\ar[d]_{r'_{Fx}} & Fx\oplus' F0\ar[d]^{F_\oplus(x,0)} \\ Fx & F(x\oplus 0)\ar[l]^{Fr_x} }\qquad 
\xymatrix{ 0'\oplus' Fx\ar[r]^-{F_0\oplus' id}\ar[d]_{l'_{Fx}} & F0\oplus' Fx\ar[d]^{F_\oplus(0,x)} \\ Fx & F(0\oplus x)\ar[l]^{Fl_x} }
\]
\end{itemize}
\end{defn}

\begin{rem}{\rm
If $\CC,\CC'$ are merely monoidal categories, with no commutators, a triple $(F,F_\oplus,F_0)$ as above satisfying (SF1) and (SF3) is called a {\em monoidal functor}. 
}
\end{rem}

\begin{defn}\label{AC-functor}
Given two AC categories $\CC,\CC'$, an {\em AC functor} is a triple $(F,F_\oplus,F_0)$ consisting of:
\begin{itemize}
\item a functor $F:\Cc\to\Cc'$;
\item a family of natural isomorphism $F_\oplus(x,y):Fx\oplus' Fy\to F(x\otimes y)$ (the {\em AC isomorphisms});
\item an isomorphism $F_0:0'\to F0$ (the {\em zero isomorphism}). 
\end{itemize}
These data must satisfy the following axioms:
\begin{itemize}
\item[{\rm (AF1)}] for every objects $x,y,z,t$ in $\Cc$ the following diagram commutes:
\[
\xymatrix{
(Fx\oplus' Fy)\oplus'(Fz\oplus' Ft)\ar[rrr]^-{b'(Fx,Fy,Fz,Ft)}\ar[d]_-{F_\oplus(x,y)\oplus' F_\oplus(z,t)}
&&&
(Fx\oplus' Fz)\oplus' (Fy\oplus' Ft)\ar[d]^-{F_\oplus(x,z)\oplus' F_\oplus(y,t)} 
\\
F(x\oplus y)\oplus' F(z\oplus t)\ar[d]_-{F_\oplus(x\oplus y,z\oplus t)}
&&&
F(x\oplus z)\oplus' F(y\oplus t)\ar[d]^-{F_\oplus(x\oplus z,y\oplus t)}
\\
F((x\oplus y)\oplus (z\oplus t))\ar[rrr]_-{Fb(x,y,z,t)} &&& F((x\oplus z)\oplus(y\oplus t)) }
\]
\item[{\rm (AF2)}] for ever object $x$ in $\Cc$ the following diagrams commute:
\[
\xymatrix{
0'\oplus' Fx\ar[r]^-{l'_{Fx}}\ar[d]_-{F_0\oplus' id} & Fx
\\
F0\oplus' Fx\ar[r]_-{F_\oplus(0,x)} & F(0\oplus x)\ar[u]_-{Fl_x} }
\quad
\xymatrix{
Fx\oplus' 0'\ar[r]^-{r'_{Fx}}\ar[d]_-{id\oplus' F_0} & Fx
\\
Fx\oplus' F0\ar[r]_-{F_\oplus(x,0)} & F(x\oplus 0)\ar[u]_-{Fr_x} }
\]
\end{itemize}
\end{defn}

Observe that axiom (SF3) in Definition~\ref{functor_monoidal_simetric} coincides with axiom (AF2) in Definition~\ref{AC-functor}. Thus, the two definitions differ only in that axioms (SF1)–(SF2) on $a,c$ are replaced by the single axiom (AF1) on $b$.

The composition of two symmetric monoidal functors $\FF:\CC\to\CC'$ and $\FF':\CC'\to\CC''$ is canonically a symmetric monoidal functor with structural isomorphisms given by
\begin{align}
(F'\circ F)_\oplus(x,y)&:=F'(F_\oplus(x,y))\circ F'_\oplus(Fx,Fy),\label{FcircF'_otimes}
\\ (F'\circ F)_0&:=F'(F_0)\circ F'_0,\label{FcircF'_1}
\end{align}
Similarly, the composition of two AC functors is canonically an AC functor with the same structural isomorphisms (see \cite{Elgueta-2025}, Proposition~2.6).

Finally, the 2-morphisms between two such morphisms are the same in both cases. Depending on the context, they are called one way or another. The precise definition is as follows.

\begin{defn}\label{trans_natural_monoidal}
If $\FF,\FF':\CC\to\CC'$ are two symmetric monoidal functors (resp. AC functors), with $\FF=(F,F_\oplus,F_0)$ and $\FF'=(F',F'_\oplus,F'_0)$, a {\em monoidal natural transformation} (resp. {\em AC natural transformation}) from $\FF$ to $\FF'$ is a natural transformation $\tau:F\Rightarrow F'$ that satisfies the following axioms:
\begin{itemize}
\item[(T1)] for every objects $x,y$ in $\Cc$ the following diagram commutes:
\[
\xymatrix{
Fx\oplus' Fy\ar[d]_{\tau_x\oplus'\tau_y}\ar[r]^-{F_\oplus(x,y)} & F(x\oplus y)\ar[d]^{\tau_{x\oplus y}}
\\
F'x\oplus' F'y\ar[r]_-{F'_\oplus(x,y)} & F'(x\oplus y)
}
\]
\item[(T2)] the following diagram commutes:
\[
\xymatrix{ & 0'\ar[ld]_{F_0}\ar[rd]^{F'_0} & \\ F0\ar[rr]_{\tau_0} && F'0 }
\]
\end{itemize}
\end{defn}

The symmetric monoidal categories, together with the symmetric monoidal functors and the monoidal natural transformations, form a 2-category $\mathbf{SMonCat}$, with the above composition of morphisms and the usual vertical and horizontal compositions of 2-morphisms. Similarly, the AC categories with the AC functors as morphisms and the AC natural transformations as 2-morphisms also form a 2-category $\mathbf{ACCat}$. The following result proves that both structures, symmetric monoidal categories and AC categories, are in fact equivalent.

\begin{thm} [\cite{Elgueta-2025}]\label{teorema_equivalencia}
$\mathbf{SMonCat}$ and $\mathbf{ACCat}$ are isomorphic 2-categories.
\end{thm}

The passage from one type of category to the other is as follows. Given a symmetric monoidal category $\CC=(\Cc,\oplus,0,a,c,l,r)$, we have an AC category $\CC_{ac}$ whose data $\Cc,\oplus,0,l,r$ are the same, and whose associo-commutator $b(x,y,z,t)$ is given by the unique isomorphism defined by the commutative diagram
\begin{equation}\label{associocommutador}
\xymatrix@C=3pc{
x\oplus(y\oplus(z\oplus t))\ar[r]^-{id\oplus a_{y,z,t}} 
&
x\oplus((y\oplus z)\oplus t)\ar[r]^-{id\oplus(c_{y,z}\cdot id)}\ar[dd]|-{a_{x,yz,t}} 
&
x\oplus((z\oplus y)\oplus t)\ar[r]^-{id\oplus a^{-1}_{z,y,t}}\ar[dd]|-{a_{x,zy,t}}
&
x\oplus(z\oplus(y\oplus t))\ar[d]|-{a_{x,z,yt}}
\\
(x\oplus y)\oplus(z\oplus t)\ar[u]|-{a^{-1}_{x,y,zt}}\ar[d]|-{a_{xy,z,t}}
&&&
(x\oplus z)\oplus(y\oplus t)
\\
((x\oplus y)\oplus z)\oplus t\ar[r]_-{a^{-1}_{x,y,z}\oplus id}
& 
(x\oplus(y\oplus z))\oplus t\ar[r]_-{(id\oplus c_{y,z})\oplus id}
&
(x\oplus(z\oplus y))\oplus t\ar[r]_-{a_{x,z,y}\oplus id}
&
((x\oplus z)\oplus y)\oplus t\ar[u]|-{a^{-1}_{xz,y,t}} 
}
\end{equation}
We shall refer to these isomorphisms as the {\em canonical associo-commutators} of $\CC$. Conversely, given an AC category $\DD=(\Dd,\oplus,0,b,l,r)$ we have a symmetric monoidal category $\DD_{sm}$ with the same data $\Dd,\oplus,0,l,r$, and the associator $a_{x,y,z}$ and commutator $c_{x,y}$ respectively given by the composite isomorphisms
\begin{equation}\label{associador_b}
\xymatrix@C=3pc{
x\oplus(y\oplus z)\ar[r]^-{r_x^{-1}\oplus id} & (x\oplus 0)\oplus(y\oplus z)\ar[r]^-{b(x,0,y,z)} & (x\oplus y)\oplus(0\oplus z)\ar[r]^-{id\oplus l_z} & (x\oplus y)\oplus z }
\end{equation}
\begin{equation}\label{commutador_b}
\xymatrix@C=3pc{
x\oplus y\ar[r]^-{l_x^{-1}\oplus r_y^{-1}} & (0\oplus x)\oplus(y\oplus 0)\ar[r]^-{b(0,x,y,0)} & (0\oplus y)\oplus(x\oplus 0)\ar[r]^-{l_y\oplus r_x} & y\oplus x }
\end{equation}
We shall refer to the isomorphisms (\ref{associador_b}) and (\ref{commutador_b}) as the {\em canonical associators} and {\em canonical commutators}, respectively, of $\DD$. 

It is shown that the assignments $\CC\mapsto\CC_{ac}$ and $\DD\mapsto\DD_{sm}$ extend to mutually inverse 2-functors
\begin{align*}
&(-)_{ac}:\mathbf{SMonCat}\to\mathbf{ACCat},
\\
&(-)_{sm}:\mathbf{ACCat}\to\mathbf{SMonCat}
\end{align*}
acting as the identity on both morphisms and 2-morphisms. Indeed, if $\FF=(F,F_\oplus,F_0)$ is a symmetric monoidal functor between symmetric monoidal categories $\CC$ and $\CC'$, the same triple $(F,F_\oplus,F_0)$ is an AC functor between the corresponding AC categories $\CC_{ac}$ and $\CC'_{ac}$, and conversely, every AC functor $\FF=(F,F_\oplus,F_0)$ between AC categories $\DD$ and $\DD'$ is also a symmetric monoidal functor between the corresponding symmetric monoidal categories $\DD_{sm}$ and $\DD'_{sm}$. The first assertion is an immediate consequence of the coherence theorem for symmetric monoidal functors (see Theorem~1.3.12 in \cite{Yau-2024-I}) and the definition of the associo-commutator in $\CC_{ac}$. As for the second assertion, whose proof details are omitted in \cite{Elgueta-2025}, it follows from the two commutative diagrams
\begin{equation}\label{diagrama1}
\xymatrix@R=3.5pc{
Fx\oplus'(Fy\oplus' Fz)\ar[r]^-{(r')^{-1}_{Fx}\oplus' id} \ar[dd]|{id\,\oplus' F_\oplus(y,z)}
&
(Fx\oplus' 0')\oplus'(Fy\oplus' Fz)\ar[r]^-{b'(Fx,0',Fy,Fz)}\ar[d]|{(id\,\oplus' F_0)\oplus' id}\ar@{}@<-5ex>[r]|{\textcircled{N}}
&
(Fx\oplus' Fy)\oplus'(0'\oplus' Fz)\ar[r]^-{id\,\oplus' l'_{Fz}}\ar[d]|-{id\oplus'(F_0\oplus' id)}
&
(Fx\oplus' Fy)\oplus' Fz\ar[dd]|-{F_\oplus(x,y)\oplus' id}
\\
&
(Fx\oplus' F0)\oplus'(Fy\oplus' Fz)\ar[r]_-{b'(Fx,F0,Fy,Fz)}\ar[d]|-{F_\oplus(x,0)\oplus' F_\oplus(y,z)}\ar@{}@<+5ex>[l]|{\textcircled{2}}
&
(Fx\oplus' Fy)\oplus'(F0\oplus' Fz)\ar[d]|-{F_\oplus(x,y)\oplus' F_\oplus(0,z)}\ar@{}@<-5ex>[r]|{\textcircled{2}}
&
\\
Fx\oplus' F(y\oplus z)\ar[r]^-{Fr^{-1}_x\oplus id}\ar[d]|-{F_\oplus(x,y\oplus z)}\ar@{}@<-5ex>[r]|{\textcircled{N}}
& 
F(x\oplus 0)\oplus' F(y\oplus z)\ar[d]|-{F_\oplus(x\oplus 0,y\oplus z)}\ar@{}[r]|{\textcircled{1}}
&
F(x\oplus y)\oplus' F(0\oplus z)\ar[r]^-{id\oplus Fl_z}\ar[d]|-{F_\oplus(x\oplus y,0\oplus z)}\ar@{}@<-5ex>[r]|{\textcircled{N}}
&
F(x\oplus y)\oplus' Fz\ar[d]|-{F_\oplus(x\oplus y,z)}
\\
F(x\oplus(y\oplus z))\ar[r]_-{F(r^{-1}_x\oplus id)}
&
F((x\oplus 0)\oplus(y\oplus z))\ar[r]_-{Fb(x,0,y,z)} 
&
F((x\oplus y)\oplus(0\oplus z))\ar[r]_-{F(id\oplus l_z)}
&
F((x\oplus y)\oplus z)
}
\end{equation}
and 
\begin{equation}\label{diagrama2}
\xymatrix@R=3.5pc@C=3pc{
Fx\oplus' Fy\ar[r]^-{(l')^{-1}_{Fx}\oplus'(r')^{-1}_{Fy}}\ar[ddd]_-{F_\oplus(x,y)}\ar@{}@<-5ex>[r]|{\textcircled{2}}
& 
(0'\oplus' Fx)\oplus'(Fy\oplus' 0')\ar[r]^-{b'(0',Fx,Fy,0')}\ar[d]|{(F_0\oplus' id)\oplus'(id\oplus' F_0)}
\ar@{}@<-5ex>[r]|{\textcircled{N}}
& 
(0'\oplus' Fy)\oplus'(Fx\oplus' 0')\ar[d]|{(F_0\oplus' id)\oplus'(id\oplus' F_0)}\ar[r]^-{l'_{Fy}\oplus' r'_{Fx}}
\ar@{}@<-5ex>[r]|{\textcircled{2}}
&
Fy\oplus' Fx\ar[ddd]^-{F_\oplus(y,x)}
\\
&
(F0\oplus' Fx)\oplus'(Fy\oplus' F0)\ar[r]_-{b'(F0,Fx,Fy,F0)}\ar[d]|{F_\oplus(0,x)\oplus'F_\oplus(y,0)}
& 
(F0\oplus' Fy)\oplus'(Fx\oplus' F0)\ar[d]|{F_\oplus(0,y)\oplus'F_\oplus(x,0)}
& 
\\
&
F(0\oplus x)\oplus'F(y\oplus 0)\ar[d]|{F_\oplus(0\oplus x,y\oplus 0)}\ar@/^1pc/[luu]|-{Fl_x\oplus' Fr_y}\ar@{}[l]|{\textcircled{N}}\ar@{}[r]|{\textcircled{1}}
& 
F(0\oplus y)\oplus'F(x\oplus 0)\ar[d]|{F_\oplus(0\oplus y,x\oplus 0)}\ar@/_1pc/[ruu]|-{\ \ Fl_x\oplus' Fr_y}\ar@{}[r]|{\textcircled{N}}
& 
\\
F(x\oplus y)\ar[r]_-{F(l^{-1}_x\oplus r^{-1}_y)}
&
F((0\oplus x)\oplus(y\oplus 0))\ar[r]_-{Fb(0,x,y,0)}
&
F((0\oplus y)\oplus (x\oplus 0))\ar[r]_-{F(l_y\oplus r_x)}
&
F(y\oplus x)
}
\end{equation}
In these diagrams, the subdiagrams labelled $\textcircled{N}$ are naturality squares, those labelled $\textcircled{1}$ commute because of axiom (AF1), and those labelled $\textcircled{2}$ because of axiom (AF2) (and the functoriality of $\oplus'$). Hence in both cases the outer diagrams commute, proving axioms (SF1) and (SF2). Axiom (SF3) follows directly from (AF2). The proof that the 2-functors so defined are mutually inverse can be found in \cite{Elgueta-2025}.

\subsection{Symmetric 2-groups and AC 2-groups}

For our purposes, the point of Theorem~\ref{teorema_equivalencia} is that it restricts to an isomorphism of 2-categories
\begin{equation}\label{iso_2SGrp-2ACGrp}
\xymatrix{
\mathbf{2SGrp}\ar@<+1.1ex>[r]^-{(-)_{ac}}_-{\cong}
& 
\mathbf{2ACGrp}\ar@<+1.1ex>[l]^-{(-)_{sm}} }
\end{equation}
where $\mathbf{2SGrp}$ is the full sub-2-category of $\mathbf{SMonCat}$ whose objects are the {\em symmetric 2-groups}, i.e. the symmetric monoidal categories whose underlying category is a groupoid, and such that every object $x$ has a (generically weak) inverse $\overline x$ with respect to $\oplus$, and $\mathbf{2ACGrp}$  is the full sub-2-category of $\mathbf{ACCat}$ whose objects are the {\em AC 2-groups}, i.e. the AC categories whose underlying category is a groupoid, and such that every object $x$ has a (generically weak) inverse $\overline x$ with respect to $\oplus$. Indeed, both 2-functors $(-)_{ac}$ and $(-)_{sm}$ preserve the data $\Cc,\oplus,0,l,r$. Hence every symmetric 2-group is mapped by $(-)_{ac}$ to an AC 2-group, and conversely, every AC 2-group is mapped by $(-)_{sm}$ to a symmetric 2-group. In this sense, both structures can be regarded as groupoid analogues of abelian groups, as they constitute different descriptions of essentially the same structure.

For what follows, we will need two facts about the 2-category $\mathbf{2SGrp}$ (and the corresponding facts for $\mathbf{2ACGrp}$): the existence of internal homs and the simplified description of morphisms and 2-morphisms.

\subsubsection{Internal homs}
The 2-category $\mathbf{2SGrp}$ has internal homs (see \cite{Dupont-2008}). In fact, we only need the existence of a canonical sum functor on the hom-groupoids. Indeed, if $\underline{Hom}(\GG,\GG')$ denotes the groupoid of morphisms between two symmetric 2-groups $\GG,\GG'$ and the 2-morphisms between them, we have a functor 
\[
\boxplus:\underline{Hom}(\GG,\GG')\times\underline{Hom}(\GG,\GG')\to\underline{Hom}(\GG,\GG')
\]
defined pointwise. More precisely, if $\FF,\FF':\GG\to\GG'$, with $\FF=(F,F_\oplus,F_0)$ and $\FF'=(F',F_\oplus',F'_0)$ , then $\FF\boxplus\FF'$ is the morphism given by the triple $(F\boxplus F',(F\boxplus F')_\oplus,(F\boxplus F')_0)$, where $F\boxplus F'$ is the functor acting on objects $x$ and morphisms $f$ by
\begin{align}
(F\boxplus F')x=Fx\oplus F'x,\label{F_oplus_F'_objectes}
\\
(F\boxplus F')f=Ff\oplus F'f,\label{F_oplus_F'_morfismes}
\end{align}
$(F\boxplus F')_\oplus(x,y)$ the natural isomorphism given by
\begin{equation}\label{varphi_oplus_varphi'}
\xymatrix@C=4.5pc{
(Fx\oplus F'x)\oplus(Fx\oplus F'y)\ar[r]^-{b(Fx,F'x,Fy,F'y)} & (Fx\oplus Fy)\oplus(F'x\oplus F'y)\ar[r]^-{F_\oplus(x,y)\oplus F'_\oplus(x,y)} & F(x\oplus y)\oplus F'(x\oplus y)
}
\end{equation}
for any objects $x,y$, with $b(a,b,c,d):(a\oplus b)\oplus(c\oplus d)\to(a\oplus c)\oplus(b\oplus d)$ the canonical isomorphism given by (\ref{associocommutador}), and $(F\boxplus F')_0$ is the isomorphism
\begin{equation}\label{F_boxplus_F'_0}
\xymatrix{
0'\ar[r]^-{d'} & 0'\oplus 0'\ar[r]^-{F_0\oplus F'_0} & F0\oplus F'0.
}
\end{equation}
The same fact remains valid for AC 2-groups. Indeed, if $\GG,\GG'$ are any two AC 2-groups, the groupoids $\underline{Hom}(\GG,\GG')$ and $\underline{Hom}(\GG_{sm},\GG'_{sm})$ coincide because $(-)_{sm}$ acts as the identity on morphisms and 2-morphisms. Now, this groupoid has a symmetric 2-group structure and hence, also an AC 2-group structure given by the corresponding canonical associo-commutator.

\subsubsection{On the description of morphisms and 2-morphisms}
\label{descripcio_morfismes}
The description of morphisms and 2-morphisms in $\mathbf{2SGrp}$ can be simplified. In fact, the simplification already occurs in the 2-category $\mathbf{2Grp}$ of {\em 2-groups}, i.e. monoidal categories whose underlying category is a groupoid, and such that every object $x$ has a (generically weak) inverse $\overline x$ with respect to $\oplus$. The morphisms in this 2-category are the monoidal functors, and the 2-morphisms the monoidal natural transformations. In this more general setting, the result can be stated as follows.

\begin{prop}\label{expressio_F_1}
Let $\GG,\GG'$ be any 2-groups. Then the following holds:
\begin{itemize}
\item[(1)] for any pair $(F,F_\oplus)$ as above satisfying (SF1) there exists one and only one isomorphism $F_0:0'\to F0$ such that $\FF=(F,F_\oplus,F_0)$ is a morphism of 2-groups, and it is given by the composite

\begin{equation}\label{expressio_varepsilon_bis}
\hspace{1.5truecm}\xymatrix@R=1.7pc@C=2.7pc{
0'\ar[r]^-{\eta} & \overline{F0}\oplus' F0\ar[r]^-{id\oplus' F(d^{-1})} &  \overline{F0}\oplus' F(0\oplus 0)\ar[d]^-{id\oplus' F_\oplus(0,0)^{-1}} & &&& 
\\
&& \overline{F0}\oplus' (F0\oplus' F0)\ar[d]^-{a'_{\overline{F0},F0,F0}} &&&
\\
&& (\overline{F0}\oplus' F0)\oplus' F0\ar[r]_-{\eta^{-1}\oplus' id} & 0'\oplus' F0\ar[r]_-{l'_{F0}} & F0
}
\end{equation}
for any inverse $\overline{F0}$ of $F0$ and any isomorphism $\eta:0'\to\overline{F0}\oplus F0$;
\item[(2)] if $\FF'=(F',F'_\oplus,F'_0):\GG\to\GG'$ is any other morphism of 2-groups and $\tau:F\Rightarrow F'$ is a natural transformation satisfying (T1) then it also satisfies (T2).
\end{itemize}
\end{prop}

Although the result is well known, a  proof is included in Appendix~A because of its importance in what follows; a less direct proof can be found in \cite{Kock-2008}.

\begin{cor}\label{expressio_F1_bis}
Let $\GG,\GG'$ be any symmetric 2-groups. Then the following holds:
\begin{itemize}
\item[(1)] for any pair $(F,F_\oplus)$ as above satisfying (SF1) and (SF2) there exists one and only one isomorphism $F_0:0'\to F0$ such that $\FF=(F,F_\oplus,F_0)$ is a morphism of symmetric 2-groups, and it is given by (\ref{expressio_varepsilon_bis});
\item[(2)] if $\FF'=(F',F'_\oplus,F'_0):\GG\to\GG'$ is any other morphism of symmetric 2-groups and $\tau:F\Rightarrow F'$ is a natural transformation satisfying (T1) then it also satisfies (T2).
\end{itemize}
\end{cor}

\begin{proof}
The additional axiom (SF2) in the definition of a morphism of {\em symmetric} 2-groups does not involve the isomorphism $F_0$. 
\end{proof}

According to Corollary~\ref{expressio_F1_bis}, a morphism between symmetric 2-groups is completely given by a pair $(F,F_\oplus)$ as above satisfying (SF1) and (SF2), the isomorphism $F_0$ being uniquely determined by the rest of the structure, and a 2-morphism between two such morphisms $(F,F_\oplus)$ and $(F',F'_\oplus)$ merely consists of a natural transformation $\tau:F\Rightarrow F'$ that satisfies (T1), axiom (T2) being automatic. 

Contrary to what one might expect at first sight, the same fact is only partially true in the 2-category of AC 2-groups. Thus, statement (2) remains true in the 2-category $\mathbf{2ACGrp}$. In fact, an AC natural transformation between two AC functors $\FF,\FF':\GG\to\GG'$ is a monoidal natural transformation between the symmetric monoidal functors $\FF,\FF':\GG_{sm}\to\GG'_{sm}$ and hence, it is completely given by a natural transformation satisfying (T1), axiom (T2) being a consequence. 

However, {\em the zero isomorphism $F_0$ in an AC functor $\FF=(F,F_\oplus,F_0)$ between two AC 2-groups $\GG,\GG'$ is not determined by the pair $(F,F_\oplus)$ satisfying axiom (AF1)}. In fact, given a pair $(F,F_\oplus)$ satisfying (AF1), the most natural approach to prove the existence of a unique $F_0:0'\to F0$ satisfying (AF2) would be to consider the corresponding symmetric 2-groups $\GG_{sm},\GG'_{sm}$ and the morphism between them defined by the same pair $(F,F_\oplus)$. By Corollary~\ref{expressio_F1_bis}, one would then conclude the existence of a unique $F_0$ satisfying (SF3) and, consequently, (AF2). The flaw in this argument is that {\em axiom (AF1) on $(F,F_\oplus)$ for the given $b,b'$ is not equivalent to axioms (SF1)-(SF2) on $(F,F_\oplus)$ for the associated pairs $(a,c),(a',c')$}. Indeed, any pair $(F,F_\oplus)$ satisfying (SF1)-(SF2) for the canonical pairs $(a,c),(a',c')$ given by (\ref{associador_b}) and (\ref{commutador_b}) also satisfies (AF1). However, as the next example shows,  $(F,F_\oplus)$ may satisfy (AF1) for the given $b,b'$ without satisfying (SF1)-(SF2) for the corresponding pairs $(a,c)$ and $(a',c')$.

\begin{ex}  {\rm (cf. \cite[Appendix]{Quang-2013}) 
Let $\Rsf$ be the ring $\ZZ[x]/\langle x^2\rangle$ of dual integers, whose elements are of the form $a+b\epsilon$, with $a,b\in\ZZ$ and $\epsilon^2=0$, and whose sum and product are given by 
\begin{align*}
(a+b\epsilon)+(a'+b'\epsilon)&=(a+a')+(b+b')\epsilon,
\\
(a+b\epsilon)\cdot(a'+b'\epsilon)&=aa'+(ab'+a'b)\epsilon.
\end{align*} 
Then an AC 2-group $\GG(\Rsf,\ZZ)$ is given as follows. The underlying groupoid $\Gg(\Rsf,\ZZ)$ has the elements in $R$ as objects, the pairs in $\ZZ\times R$ as morphisms, with 
\[
(n,a+b\epsilon):a+b\epsilon\to a+b\epsilon,
\]
and composition is given by the sum in $\ZZ$. The identity morphisms are given by the pairs $(0,a+b\epsilon)$. The functor $\oplus$ is given by the sum in $\Rsf$ on objects, and the componentwise sum in $\ZZ\times\Rsf$ on morphisms. In particular, it is strictly associative and commutative, it has the element $0\in R$ as a strict identity object, and all inverses are strict. Hence $\GG(\Rsf,\ZZ)=(\Gg(\Rsf,\ZZ),\oplus,0,1,1,1)$ is a totally strict AC 2-group. Then, for any $a,b\in\ZZ$ let $F(a,b):\Gg(\Rsf,\ZZ)\to\Gg(\Rsf,\ZZ)$ be the functor acting on objects and morphisms by
\begin{align*}
F(a,b)(x+y\epsilon)&:=(a+b\epsilon)\cdot(x+y\epsilon)=ax+(ay+bx)\epsilon,
\\
F(a,b)(n,x+y\epsilon)&:=(an,ax+(ay+bx)\epsilon)
\end{align*}
(functoriality follows from the distributivity in $\ZZ$), and
\[
F(a.b)_\oplus(x+y\epsilon,x'+y'\epsilon):F(a,b)(x+y\epsilon)\oplus F(a,b)(x'+y'\epsilon)\to F(a,b)((x+y\epsilon)\oplus(x'+y'\epsilon))
\]
be the morphism given by
\[
F(a,b)_\oplus(x+y\epsilon,x'+y'\epsilon):=(b(x+x'),a(x+x')+[a(y+y')+b(x+x')]\epsilon).
\]
It is easy to check that these isomorphisms are natural in $x+y\epsilon,x'+y'\epsilon$ and that $(F(a,b),F(a,b)_\oplus)$ satisfies (AF1). Precisely, this condition amounts to the equality
\[
b[(x+x')+(x''+x''')]=b(x+x')+b(x''+x''')
\]
for every $x,x',x'',x''\in\ZZ$. However, although $(F(a,b),F(a,b)_\oplus)$ satisfies (SF2), it does not satisfy (SF1) if $b\neq 0$. Specifically, the reader may easily check that (SF1) amounts to the condition
\[
b(x-x')=0
\]
for every $x,x'\in\ZZ$. }
\end{ex}

The problem becomes clear upon examination of diagrams (\ref{diagrama1}) and (\ref{diagrama2}). It follows from (\ref{diagrama1}) that, for a pair $(F,F_\oplus)$ satisfying (AF1) to also satisfy (SF1), we explicitly require the existence of the zero isomorphism $F_0$ satisfying (AF2), and similarly for $(F,F_\oplus)$ to satisfy (SF2), as it follows from (\ref{diagrama2}). Hence, in the above argument, the existence of the isomorphism $F_0$ satisfying (AF2) is proved by assuming it beforehand.

In fact, this is not a proof that $F_0$ is not determined by the pair $(F,F_\oplus)$ satisfying axiom (AF1). It merely rules out the argument of passing to the setting of symmetric 2-groups and using that the result holds in this setting. However, the previous example also shows that $F_0$ may certainly fail to exist for some pairs $(F,F_\oplus)$ satisfying axiom (AF1). Indeed, since the unitors in $\GG(\Rsf,\ZZ)$ are identities, axiom (AF2) for the above pair $(F(a,b),F(a,b)_\oplus)$ reduces to the equalities
\begin{align*}
F_0\oplus id_{F(a,b)(x+y\epsilon)}&:=F(a,b)_\oplus(0,x+y\epsilon)^{-1},
\\
id_{F(a,b)(x+y\epsilon)}\oplus F_0&:=F(a,b)_\oplus(x+y\epsilon,0)^{-1}
\end{align*}
for every object $x+y\epsilon$. Hence an isomorphism $F_0$ satisfying (AF2) for this pair corresponds to an integer $n$ such that
\[
n=-bx
\]
for every $x\in\ZZ$, and such an integer does not exist if $b\neq 0$.

\section{Notions of 2-ring}

As stated in the introduction, by a 2-ring we mean a groupoid equipped with a structure similar to that of a ring, but with all axioms holding only up to natural isomorphisms satisfying their own axioms. In particular, the underlying abelian group is replaced by a symmetric 2-group or an AC 2-group. Depending on this choice, and on the axioms required on the natural isomorphisms, several notions of 2-ring arise. We present three such definitions: the definitions due to Quang and to Jibladze--Pirashvili, and a third definition, equivalent to Quang's definition, that looks like the right definition of a 2-ring in the spirit of Jibladze-Pirahsvili. As pointed out by Quang et al. \cite{Quang-Hanh-Thuy-2008}, in this new definition additional data satisfying the appropriate axioms is required.

Let us begin with the definition introduced in \cite{Quang-1987} by Quang, who calls the corresponding structure an Ann-category; we shall refer to it as a “Quang 2-ring”.

\begin{defn}\label{def_2-anell_Quang}
A {\em Quang 2-ring} is a 5-tuple $\mathbb R=(\Rr,(+,0,a^+,c,l^+,r^+),(\cdot,1,a^\times,l^\times,r^\times),d,e)$ consisting of:
\begin{itemize}
\item a groupoid $\Rr$;
\item a symmetric 2-group structure $(+,0,a^+,c,l^+,r^+)$ on $\Rr$;
\item an additional monoidal structure $(\cdot,1,a^\times,l^\times,r^\times)$ on $\Rr$ (for short, the symbol $\cdot$ between objects is omitted);
\item two families of natural isomorphisms (the {\em left} and {\em right distributors})
\begin{align*}
d_{x,y,z}&:xy+xz\to x(y+z),
\\
e_{x,y,z}&:xz+yz\to (x+y)z.
\end{align*}
\end{itemize}
Moreover, these data must satisfy the following axioms:
\begin{enumerate}
\item[(2R1)] the diagrams
\begin{equation*}
\xymatrix@C=3pc{
x(y+(z+ t))\ar[r]^{id\,\cdot\,a^+_{y,z,t}}
& 
x((y+ z)+ t)
\\ 
xy+ x(z+ t)\ar[u]^{d_{x,y,z+ t}} 
& 
x(y+ z)+ xt\ar[u]_{d_{x,y+ z,t}} 
\\ 
xy+ (xz+ xt)\ar[r]_{a^+_{xy,xz,xt}} \ar[u]^{id\,+\,d_{x,z,t}}
& 
(xy+ xz)+ xt\ar[u]_{d_{x,y,z}+\, id}  
}
\quad 
\xymatrix@C=3pc{
(x+(y+ z))t\ar[r]^{a^+_{x,y,z}\,\cdot\, id} 
& 
((x+ y)+ z)t
\\ 
xt+ (y+ z)t \ar[u]^{e_{x,y+ z,t}}
& 
(x+ y)t+ zt \ar[u]_{e_{x+ y,z,t}} 
\\ 
xt+(yt+ zt)\ar[r]_{a^+_{xt,yt,zt}} \ar[u]^{id\,+\,e_{y,z,t}}
&
(xt+ yt)+ zt\ar[u]_{e_{x,y,t}\,+\,id}
}
\end{equation*}
\begin{equation*}
\xymatrix@C=3pc{
x(y+ z)\ar[r]^{id\,\cdot c_{y,z}}
& 
x(z+y)
\\ 
xy+xz\ar[r]_{c_{xy,xz}} \ar[u]^{d_{x,y,z}} 
& 
xz+ xy\ar[u]_{d_{x,z,y}} }
\quad 
\xymatrix@C=3pc{
(y+ z)x\ar[r]^{c_{y,z}\,\cdot id}
&
 (z+ y)x
\\ 
yx+ zx\ar[r]_{c_{yx,zx}} \ar[u]^{e_{y,z,x}} 
& 
zx+ yx\ar[u]_{e_{z,y,x}} }
\end{equation*}
commute for every objects $x,y,z,t\in\Rr$;

\item[(2R2)] the diagram
\begin{equation*}
\xymatrix@C=3.5pc{
(x+ y)(z+ t)\ar[r]^-{e_{x,y,z+ t}}\ar[d]_{d_{x+ y,z,t}} 
& 
x(z+ t)+ y(z+ t)\ar[r]^{d_{x,z,t}\,+\,d_{y,z,t}} 
& 
(xz+ xt)+(yz+ yt)\ar[d]^{b(xz,xt,yz,yt)}  
\\
(x+ y)z+(x+ y)t\ar[rr]_{\ \ e_{x,y,z}\,+\,e_{x,y,t}} 
& & 
(xz+ yz)+(xt+ yt)}
\end{equation*}
commutes for all objects $x,y,z,t\in\Rsc$, where $b(a,b,c,d):(a+ b)+(c+ d)\stackrel{\cong}{\to}(a+ c)+(b+ d)$ is the canonical associo-commutator of the symmetric 2-group $(\Rr,+,0,a^+,c,l^+,r^+)$ given by (\ref{associocommutador});

\item[(2R3)]
the diagram
\begin{equation*}
\xymatrix{
x(y(z+ t))\ar[rr]^{id\,\cdot d_{y,z,t}}\ar[d]_{a^\times_{x,y,z+ t}} 
&& 
x(yz+ yt)\ar[rr]^-{d_{x,yz,yt}} 
&& 
x(yz)+ x(yt)\ar[d]^{a^\times_{x,y,z}\,+\, a^\times_{x,y,t}} 
\\ 
(xy)(z+ t)\ar[rrrr]_{d_{xy,z,t}} 
&&&& 
(xy)z+ (xy)t }
\end{equation*}
commutes for every objects $x,y,z,t$;

\item[(2R4)]
the diagram
\begin{equation*}
\xymatrix{
x((z+ t)y)\ar[rr]^{id\,\cdot\ e_{z,t,y}}\ar[d]_{a^\times_{x,z+ t,y}} 
&& 
x(zy+ ty)\ar[rr]^{d_{x,zy,ty}} 
&& 
x(zy)+ x(ty)\ar[d]^{a^\times_{x,z,y}\,+\,a^\times_{x,t,y}}  
\\ 
(x(z+ t))y\ar[rr]_{d_{x,z,t}\,\cdot\, id} 
&& 
(xz+ xt)y\ar[rr]_{e_{xz,xt,y}} 
&& 
(xz)y+ (xt)y}
\end{equation*}
commutes for every objects $x,y,z,t$;

\item[(2R5)] the diagram
\begin{equation*}
\xymatrix{
((t+ z)y)x\ar[rr]^{e_{t,z,y}\,\cdot id} 
&& 
(ty+ zy)x\ar[rr]^-{e_{ty,zy,x}} 
&& 
(ty)x+ (zy)x 
\\ 
(t+ z)(yx)\ar[u]^{a^\times_{t+ z,y,x}}\ar[rrrr]_{e_{t,z,yx}} 
&&&& 
t(yx)+ z(yx)\ar[u]_{a^\times_{t,y,x}\,+\,a^\times_{z,y,x}} }
\end{equation*}
commutes for every objects $x,y,z,t$;

\item[(2R6)]
the diagrams
\begin{equation*}
\xymatrix{
1(x+ y)\ar[rr]^{d_{1,x,y}}\ar[rd]_{l^\times_{x+ y}}  
& & 
1 x+ 1 y\ar[ld]^{l^\times_x\,+\,l^\times_y} 
\\ &
x+ y 
&} 
\quad 
\xymatrix{
(x+ y) 1\ar[rr]^{e_{x,y,1}}\ar[rd]_{r^\times_{x+ y}}  
& & 
x 1+ y 1\ar[ld]^{r^\times_x\,+\,r^\times_y} 
\\ 
& 
x+ y
&}
\end{equation*}
commute for every objects $x,y$.
\end{enumerate}
\end{defn}

\begin{rem}{\rm
Axiom (2R1) corresponds to Quang's axiom (Ann-1) in \cite{Quang-1987}, (2R2)-(2R5) to (Ann-2), and (2R6) to (Ann-3). }
\end{rem}

\begin{notacio}{\rm 
The underlying symmetric 2-group $(\Rr,+,0,a^+,c,l^+,r^+)$ of a 2-ring $\RR$ will be denoted by $\RR_+$. }
\end{notacio}

A Jibladze--Pirashvili categorical ring, which we call a {\em Jibladze--Pirashvili 2-ring}, is defined by the same data and axioms except that axiom (2R1) is replaced by a new axiom on the canonical associo-commutators of the underlying symmetric 2-group.  

\begin{defn}\label{def_JP-2-ring}
A {\em Jibladze--Pirashvili 2-ring} is a 5-tuple $\mathbb R=(\Rr,(+,0,a^+,c,l^+,r^+),(\cdot,1,a^\times,l^\times,r^\times),d,e)$ consisting of:
\begin{itemize}
\item a groupoid $\Rr$;
\item a symmetric 2-group structure $(+,0,a^+,c,l^+,r^+)$ on $\Rr$;
\item a monoidal structure $(\cdot,1,a^\times,l^\times,r^\times)$ on $\Rr$;
\item two families of natural isomorphisms (the {\em left} and {\em right distributors})
\begin{align*}
d_{x,y,z}&:xy+xz\to x(y+z),
\\
e_{x,y,z}&:xz+yz\to (x+y)z.
\end{align*}
\end{itemize}
Moreover, these data must satisfy axioms (2R2)-(2R6) as in Definition~\ref{def_2-anell_Quang} together with the following axiom:
\begin{enumerate}
\item[(2R1)$'$] the diagrams
\begin{equation*}
\xymatrix@C=6pc{
(xy+xz)+(xt+ xu))\ar[r]^{b(xy,xz,xt,xu)}\ar[d]_{d_{x,y,z}+d_{x,t,u}} 
& 
(xy+ xt)+(xz+ xu)\ar[d]^{d_{x,y,t}+d_{x,z,u}} 
\\  
x(y+z)+ x(t+u)\ar[d]_{d_{x,y+z,t+u}} 
& 
x(y+t)+ x(z+u)\ar[d]^{d_{x,y+t,z+u}} 
\\  
x[(y+z)+(t+u)]\ar[r]_{id\cdot b(y,z,t,u)}
& 
x[(y+t)+(z+u)]
}
\end{equation*}
\begin{equation*} 
\xymatrix@C=6pc{
(xu+ yu)+(zu+ tu)\ar[r]^{b(xu,yu,zu,tu)}\ar[d]_{e_{x,y,u}+e_{z,t,u}} 
& 
(xu+zu)+(yu+tu)\ar[d]^{e_{x,z,u}+e_{y,t,u}} 
\\ 
(x+y)u+ (z+t)u\ar[d]_{e_{x+y,z+t,u}} 
& 
(x+z)u+ (y+t)u\ar[d]^{e_{x+z.y+t,u}} 
\\  
[(x+y)+(z+t)]u\ar[r]_{b(x,y,z,t)\cdot id} 
& 
[(x+z)+(y+t)]u}
\end{equation*}
commute for every objects $x,y,z,t,u$.
\end{enumerate}
\end{defn}

\begin{rem}{\rm 
The first three commutative diagrams in \cite{Jibladze-Pirashvili-2007} correspond to axioms (2R3)-(2R5), the fourth to (2R6), the fifth and seventh to (2R1)$'$ and the sixth to (2R2).}
\end{rem}

Observe that axiom (2R1) in Definition~\ref{def_2-anell_Quang} corresponds to the condition that the two pairs $(x\cdot-,d_{x,-,-})$ and $(-\cdot z,e_{-,-,z})$, for all objects $x,z$, satisfy axioms (SF1)--(SF2), while axiom (2R1)$'$ in Definition~\ref{def_JP-2-ring} corresponds to the condition that the same two pairs satisfy axiom (AF1). Consequently, according to our discussion in \S~\ref{descripcio_morfismes}, the two definitions are not equivalent: every Quang 2-ring is a Jibladze--Pirashvili 2-ring, but not conversely. 

To find an appropriate definition of 2-ring in the spirit of Jibladze and Pirashvili, we reformulate the axioms in Quang's definition so as to express them in terms of morphisms and 2-morphisms in the 2-category $\mathbf{2SGrp}$. Concretely, it follows from Proposition~\ref{expressio_F_1} that axiom (2R1) can be equivalently formulated as follows:

\begin{itemize}
\item[(2R1)] {\em for every objects $x,z$ the two pairs $(x\cdot-,d_{x,-,-})$ and $(-\cdot z,e_{-,-,z})$ are endomorphisms of the symmetric 2-group $\RR_+$.}
\end{itemize}
This allows one to regard $d$ as endowing each functor $x\cdot-$, for every object $x$, with the structure of a symmetric 2-group endomorphism $d_{x,-,-}$, and similarly for $e$ and the functors $-\cdot z$. Thus, the axiom can be incorporated into the data of a Quang 2-ring. The resulting endomorphisms $(x\cdot-,d_{x,-,-})$ and $(-\cdot z,e_{-,-,z})$ will be denoted by $\mathbb{x}\cdot-$ and $-\cdot\mathbb{z}$, respectively.

The remaining axioms express the condition that certain natural isomorphisms between endomorphisms of $\RR_+$ satisfy (T1) and, therefore, by Proposition~\ref{expressio_F_1}, are monoidal (i.e. 2-morphisms in $\mathbf{2SGrp}$). More precisely, for any objects $x,y,z$ we have the endomorphisms of $\RR_+$
\[
(\mathbb x\cdot -)\boxplus(\mathbb y\cdot -),\,(\mathbb x+\mathbb y)\cdot -,\,(-\cdot \mathbb y)\boxplus(-\cdot \mathbb z),\,-\cdot(\mathbb y+\mathbb z),
\]
with $(\mathbb x\cdot -)\boxplus(\mathbb y\cdot -)$ and $(-\cdot \mathbb y)\boxplus(-\cdot \mathbb z)$ defined by (\ref{F_oplus_F'_objectes})-(\ref{varphi_oplus_varphi'}), and the natural isomorphisms
\begin{align*}
d_{x,y,-}&:(x\cdot -)\boxplus(y\cdot -)\Rightarrow (x+y)\cdot -,
\\
e_{-,y,z}&:(-\cdot y)\boxplus(-\cdot z)\Rightarrow -\cdot(y+z),
\end{align*}
on the one hand, and the endomorphisms of $\RR_+$
\[
(\mathbb x\cdot -)\circ(\mathbb y\cdot -),\,(\mathbb x\mathbb y)\cdot -,\,(\mathbb x\cdot -)\circ(-\cdot \mathbb z),\,(-\cdot \mathbb z)\circ(\mathbb x\cdot -),\,(-\cdot \mathbb y)\circ(-\cdot\mathbb z),\,-\cdot(\mathbb y\mathbb z),
\]
where the composite endomorphisms are defined by (\ref{FcircF'_otimes}), and the natural isomorphisms
\begin{align*}
a^\times_{x,y,-}&:(x\cdot -)\circ(y\cdot -)\Rightarrow (xy)\cdot -,
\\
a^\times_{x,-,z}&:\,(xy)\cdot -,\,(x\cdot -)\circ(-\cdot z)\Rightarrow(-\cdot z)\circ(x\cdot -),
\\
a^\times_{-,y,z}&:(-\cdot y)\circ(-\cdot z)\Rightarrow -\cdot(yz),
\end{align*}
on the other. Then axioms (2R2)-(2R5) can be equivalently formulated as follows:

\begin{itemize}
\item[(2R2)] {\em $d_{x,y,-}$ is a monoidal natural isomorphism for every objects $x,y$ (equivalently, $e_{-,y,z}$ is a monoidal natural isomorphism for every objects $y,z$);}
\item[(2R3)] {\em $a^\times_{x,y,-}$ is a monoidal natural isomorphism for every objects $x,y$;}
\item[(2R4)] {\em $a^\times_{x,-,z}$ is a monoidal natural isomorphism for every objects $x,y$;}
\item[(2R5)] {\em $a^\times_{-,y,z}$ is a monoidal natural isomorphism for every objects $x,y$.}
\end{itemize}
Finally, axiom (2R6) can be equivalently formulated as follows:

\begin{itemize}
\item[(2R6)] {\em $l^\times_-:\mathbb 1\cdot -\Rightarrow id_{\RR_+}$ and $r^\times_-:-\cdot\mathbb 1\Rightarrow id_{\RR_+}$ are monoidal natural isomorphisms.}
\end{itemize}

Consequently, we arrive at the following equivalent but more conceptual description of a Quang 2-ring.

\begin{prop}
A {\em Quang 2-ring} is given by a 5-tuple $\mathbb R=(\Rr,(+,0,b,l^+,r^+),(\cdot,1,a^\times,l^\times,r^\times),d,e)$ consisting of:
\begin{itemize}
\item a groupoid $\Rr$;
\item a symmetric 2-group structure $(+,0,b,l^+,r^+)$ on $\Rr$;
\item a monoidal structure $(\cdot,1,a^\times,l^\times,r^\times)$ on $\Rr$;
\item a structure of symmetric 2-group endomorphism $d_{x,-,-}$ on each functor $x\cdot-$ for every object $x$;
\item a structure of symmetric 2-group endomorphism $e_{-,-,z}$ on each functor $-\cdot z$ for every object $z$.
\end{itemize}
Moreover, these data must satisfy the following axioms for every objects $x,y,z$ in $\Rr$:
\begin{enumerate}
\item[(2R2)] $d_{x,y,-}:(\mathbb x\cdot -)\boxplus(\mathbb y\cdot -)\Rightarrow (\mathbb x+\mathbb y)\cdot -$ (equivalently, $e_{-,y,z}:(-\cdot\mathbb y)\boxplus(-\cdot\mathbb z)\Rightarrow -\cdot(\mathbb y+\mathbb z)$) is a monoidal natural isomorphism;)
\item[(2R3)] $a^\times_{x,y,-}:(\mathbb x\cdot -)\circ(\mathbb y\cdot -)\Rightarrow (\mathbb x\mathbb y)\cdot -$ is a monoidal natural isomorphism;
\item[(2R4)] $a^\times_{x,-,z}:\,(\mathbb x\mathbb y)\cdot -,\,(\mathbb x\cdot -)\circ(-\cdot\mathbb z)\Rightarrow(-\cdot\mathbb z)\circ(\mathbb x\cdot -)$ is a monoidal natural isomorphism;
\item[(2R5)] $a^\times_{-,y,z}:(-\cdot\mathbb y)\circ(-\cdot\mathbb z)\Rightarrow -\cdot(\mathbb y\mathbb z)$ is a monoidal natural isomorphism;
\item[(2R6)] $l^\times_-:\mathbb 1\cdot -\Rightarrow id_{\RR_+}$ and $r^\times_-:-\cdot\mathbb 1\Rightarrow id_{\RR_+}$ aare monoidal natural isomorphisms.
\end{enumerate}
\end{prop}

We may now replace the terms ``symmetric 2-group'' and ``monoidal natural isomorphism'' throughout with the notions of ``AC 2-group'' and ``AC natural isomorphism''. Taking into account the discussion of morphisms between AC 2-groups in \S~\ref{descripcio_morfismes}, we obtain the following alternative definition of a 2-ring.

\begin{defn}\label{def_conceptual_AC-2-ring}
An {\em AC 2-ring} is a 5-tuple $\mathbb R=(\Rr,(+,0,b,l^+,r^+),(\cdot,1,a^\times,l^\times,r^\times),(d,m),(e,n))$ consisting of:
\begin{itemize}
\item a groupoid $\Rr$;
\item an AC 2-group structure $(+,0,b,l^+,r^+)$ on $\Rr$;
\item a monoidal structure $(\cdot,1,a^\times,l^\times,r^\times)$ on $\Rr$;
\item a structure of AC 2-group endomorphism $(d_{x,-,-},m_x)$ on each functor $x\cdot-$ for every object $x$;
\item a structure of AC 2-group endomorphism $(e_{-,-,z},n_z)$ on each functor $-\cdot z$ for every object $z$.
\end{itemize}
Moreover, these data must satisfy the following axioms for every objects $x,y,z$ in $\Rr$ ($\mathbb x\cdot-$ for every object $x$ denotes the endomorphism $(x\cdot-,d_{x,-,-},m_x)$ of the AC 2-group $\RR_+=(\Rr,+,0,b,l^+,r^+)$, and similarly $-\cdot\mathbb z$ for every object $z$, and the sum and composition of endomorphisms are also given by (\ref{F_oplus_F'_objectes})-(\ref{varphi_oplus_varphi'}) and (\ref{FcircF'_otimes}), respectively):
\begin{enumerate}
\item[(2R2)$'$] $d_{x,y,-}:(\mathbb x\cdot -)\boxplus(\mathbb y\cdot -)\Rightarrow (\mathbb x+\mathbb y)\cdot -$ (equivalently, $e_{-,y,z}:(-\cdot\mathbb y)\boxplus(-\cdot\mathbb z)\Rightarrow -\cdot(\mathbb y+\mathbb z)$ is an AC natural isomorphism;)
\item[(2R3)$'$] $a^\times_{x,y,-}:(\mathbb x\cdot -)\circ(\mathbb y\cdot -)\Rightarrow (\mathbb x\mathbb y)\cdot -$ is an AC natural isomorphism;
\item[(2R4)$'$] $a^\times_{x,-,z}:\,(\mathbb x\mathbb y)\cdot -,\,(\mathbb x\cdot -)\circ(-\cdot\mathbb z)\Rightarrow(-\cdot\mathbb z)\circ(\mathbb x\cdot -)$ is an AC natural isomorphism;
\item[(2R5)$'$] $a^\times_{-,y,z}:(-\cdot\mathbb y)\circ(-\cdot\mathbb z)\Rightarrow -\cdot(\mathbb y\mathbb z)$ is an AC natural isomorphism;
\item[(2R6)$'$] $l^\times_-:\mathbb 1\cdot -\Rightarrow id_{\RR_+}$ and $r^\times_-:-\cdot\mathbb 1\Rightarrow id_{\RR_+}$ are AC natural isomorphisms.
\end{enumerate}
\end{defn} 

It readily follows from the isomorphism of 2-categories (\ref{iso_2SGrp-2ACGrp}) that Quang 2-rings and AC 2-rings are essentially the same structure. Indeed, every Quang 2-ring $\RR$ induces an AC 2-ring $\RR_{ac}$: the underlying AC 2-group is $(\RR_+)_{ac}$, while the monoidal structure and the distributors remain unchanged. The required axioms hold because every endomorphism of $\RR_+$ is also an endomorphism of $(\RR_+)_{ac}$, and every monoidal natural transformation between symmetric monoidal functors is also an AC natural transformation when regarded as a transformation between the corresponding AC functors. 

Conversely, every AC 2-ring $\RR'$ induces a Quang 2-ring $\RR'_{sm}$ with the same monoidal structure and distributors, and with underlying symmetric 2-group $(\RR'_+)_{sm}$. Again, the required axioms hold because every endomorphism of $\RR'_+$ is also an endomorphism of $(\RR'_+)_{sm}$, and every AC natural transformation is monoidal. Moreover, both composite maps are the identity because both are the identity as maps between symmetric and AC 2-groups, and the monoidal structure and the distributors are unchanged. Consequently, we obtain the following.

\begin{prop}\label{equivalencia_Q-2-anell_AC-2-anell}
There is a bijective correspondence between Quang 2-rings and AC 2-rings.
\end{prop}

When the axioms in Definition~\ref{def_conceptual_AC-2-ring} are expressed in terms of commutative diagrams, we obtain the following description of an AC 2-ring or, equivalently, a Quang 2-ring.

\begin{prop}\label{def_AC-2-ring}
An AC 2-ring is given by a 7-tuple $\mathbb R=(\Rr,(+,0,b,l^+,r^+),(\cdot,1,a^\times,l^\times,r^\times),d,e,m,n)$ consisting of:
\begin{itemize}
\item a groupoid $\Rr$;
\item an AC 2-group structure $(+,0,b,l^+,r^+)$ on $\Rr$;
\item a monoidal structure $(\cdot,1,a^\times,l^\times,r^\times)$ on $\Rr$;
\item two families of natural isomorphisms (the {\em left} and {\em right distributors})
\begin{align*}
d_{x,y,z}&:xy+xz\to x(y+z),
\\
e_{x,y,z}&:xz+yz\to (x+y)z;
\end{align*}
\item two families of natural isomorphisms (the {\em left} and {\em right absorbing isomorphisms})
\begin{align*}
m_{x}&:0\to x0,
\\
n_{x}&:0\to 0x.
\end{align*}
\end{itemize}
Moreover, these data must satisfy axioms (2R2)--(2R6) as in Definition~\ref{def_2-anell_Quang} together with the following axiom:
\begin{enumerate}
\item[(2R1)$''$] the diagrams
\begin{equation*}
\xymatrix@C=6pc{
(xy+xz)+(xt+ xu))\ar[r]^{b(xy,xz,xt,xu)}\ar[d]_{d_{x,y,z}+d_{x,t,u}} 
& 
(xy+ xt)+(xz+ xu)\ar[d]^{d_{x,y,t}+d_{x,z,u}} 
\\  
x(y+z)+ x(t+u)\ar[d]_{d_{x,y+z,t+u}} 
& 
x(y+t)+ x(z+u)\ar[d]^{d_{x,y+t,z+u}} 
\\  
x[(y+z)+(t+u)]\ar[r]_{id\cdot b(y,z,t,u)}
& 
x[(y+t)+(z+u)]
}
\end{equation*}
\begin{equation*} 
\xymatrix@C=6pc{
(xu+ yu)+(zu+ tu))\ar[r]^{b(xu,yu,zu,tu)}\ar[d]_{e_{x,y,u}+e_{z,t,u}} 
& 
(xu+zu)+(yu+tu)\ar[d]^{e_{x,z,u}+e_{y,t,u}} 
\\  
(x+y)y+ (z+t)u\ar[d]_{e_{x+y,z+t,u}} 
& 
(x+z)u+ (y+t)u\ar[d]^{e_{x+z.y+t,u}} 
\\  
[(x+y)+(z+t)]u\ar[r]_{b(x,y,z,t)\cdot id} 
& 
[(x+z)+(y+t)]u}
\end{equation*}
\[
\xymatrix{
0+xy\ar[r]^-{l^+_{xy}}\ar[d]_-{m_x+id} & xy
\\
x0+xy\ar[r]_-{d_{x,0,y}} & x(0+y)\ar[u]_-{id\cdot l^+_y} }
\quad
\xymatrix{
xy+0\ar[r]^-{r^+_{xy}}\ar[d]_-{id+m_x} & xy
\\
xy+x0\ar[r]_-{d_{x,y,0}} & x(y+0)\ar[u]_-{id\cdot r^+_y} }
\]
\[
\xymatrix{
0+xy\ar[r]^-{l^+_{xy}}\ar[d]_-{n_y+id} & xy
\\
0y+xy\ar[r]_-{e_{0,x,y}} & (0+x)y\ar[u]_-{l^+_x\cdot id} }
\quad
\xymatrix{
xy+0\ar[r]^-{r^+_{xy}}\ar[d]_-{id+n_y} & xy
\\
xy+0y\ar[r]_-{e_{x,0,y}} & (x+0)y\ar[u]_-{r^+_x\cdot id} }
\]
commute for every objects $x,y,z,t,u$.
\end{enumerate}
\end{prop}

\begin{proof}
Being an AC natural isomorphism is the same as being a monoidal natural isomorphism, so that conditions (2R2)$'$--(2R6)$'$ coincide with conditions (2R2)--(2R6). Axiom (2R1)$''$ expresses the fact that $(x\cdot,d_{x,-,-},m_x)$ and $(-\cdot z,e_{-,-,z},n_z)$ are AC functors (hence, endomorphisms of the AC 2-group $\RR_+$).
\end{proof}

Since giving an AC 2-group structure on $\Rr$ is equivalent to giving a symmetric 2-group structure, we recover Theorem~4 in \cite{Quang-Hanh-Thuy-2008}, which states that every Jibladze--Pirashvili 2-ring for which there exist left and right absorbing isomorphisms as above making the last four diagrams in (2R1)$''$ commute is a Quang 2-ring. In fact, we have shown that this correspondence is bijective.

\appendix

\section{On the homomorphisms of 2-groups, and the 2-morphisms between them}

Let $\GG=(\Gg,\otimes,1,a,l,r)$ and $\GG'=(\Gg',\otimes',1',a',l',r')$ be two 2-groups. By definition, a homomorphism of 2-groups from $\GG$ to $\GG'$ is a monoidal functor between them as monoidal groupoids. Therefore it consists of a triple $\FF=(F,F_\otimes,F_1)$ with $F:\Gg\to\Gg'$ a functor, $F_\otimes=\{F_\otimes(x,y):Fx\otimes' Fy\to F(x\otimes y),\,x,y\in\Gg\}$ a family of natural isomorphisms, and $F_1:1'\to F1$ an isomorphism, all these data satisfying axioms (SF1) and (SF3) in Definition~\ref{functor_monoidal_simetric}. The purpose of this Appendix is to prove that the isomorphism $F_1$ is uniquely determined by the data $(F,F_\otimes)$ and the required axioms. In fact, the result is true in the following more general situation. 

\begin{prop}\label{unicitat_varepsilon}
Let $\CC=(\Cc,\otimes,1,a,l,r)$ be a monoidal category and $\GG=(\Gg,\otimes',1',a',l',r')$ a 2-group. Then:
\begin{itemize}
\item[(1)] for any pair $(F,F_\otimes)$, with $F:\Gg\to\Gg'$ a functor and $F_\otimes$ a family of natural isomorphisms $F_\otimes(x,y):Fx\otimes' Fy\to F(x\otimes y)$ satisfying (SF1), there exists one and only one isomorphism $F_1:1'\to F1$ such that $\FF=(F,F_\oplus,F_1)$ is a morphism of 2-groups, and it is given by the composite$F_1$ is given by the composite morphism
\begin{equation}\label{expressio_varepsilon}
\hspace{1truecm}\xymatrix@R=1.7pc@C=2.7pc{
1'\ar[r]^-{\eta} & (F1)^{*}\otimes' F1\ar[r]^-{id\otimes' F(d^{-1})} &  (F1)^{*}\otimes' F(1\otimes 1)\ar[d]^-{id\otimes' F_\otimes(1,1)^{-1}} & &&& 
\\
&& (F1)^{*}\otimes' (F1\otimes' F1)\ar[d]^-{a_{(F1)^{*},F1,F1}} &&&
\\
&& ((F1)^{*}\otimes' F1)\otimes' F1\ar[r]_-{\eta^{-1}\otimes' id} & 1'\otimes' F1\ar[r]_-{l_{F1}} & F1
}
\end{equation}
for any (weak) inverse $(F1)^{*}$ of $F1$, and any isomorphism $\eta:1'\to (F1)^{*}\otimes' F1$;
\item[(2)] if $\FF'=(F',F'_\otimes,F'_1):\CC\to\GG$ is any other monoidal functor and $\tau:F\Rightarrow F'$ is a natural transformation satisfying axiom (T1) in Definition~\ref{trans_natural_monoidal} then it also satisfies (T2).
\end{itemize}
\end{prop}

\begin{proof}
(1) The functors $-\otimes' Fx,Fx\otimes'-:\Gg\to\Gg$ for any object $x$ in $\Cc$ are equivalences of categories, with (weak) inverses respectively given by the functors $-\otimes'(Fx)^*,(Fx)^*\otimes'-$ for any (weak) inverse of $Fx$. It follows that the isomorphism $F_1$ satisfying (SF3) is uniquely determined by either of the morphisms $id_{Fx}\otimes F_1$ or $F_1\otimes id_{Fx}$. Hence, if it exists, it is necessarily unique. To prove that, for any pair $(F,F_\otimes)$ satisfying (SF1), the morphism given by (\ref{expressio_varepsilon}) satisfies (SF3) notice first that a morphism $F_1$ satisfies (SF3) if and only if it makes the diagram
\begin{equation}\label{varepsilon_cond1'}
\xymatrix{
F1\otimes'1'\ar[r]^{id\otimes' F_1}\ar[d]_{r'_{F1}} & F1\otimes' F1\ar[d]^{F_\otimes(1,1)}
\\
F1 & F(1\otimes 1)\ar[l]^{Fd}
}
\end{equation}
commute, where $d:1\otimes 1\to 1$ is the basic unitor of $\CC$. Clearly, this diagram commutes if $F_1$ satisfies (SF3). Conversely, since each functor $-\otimes' Fx$ is an equivalence, the commutativity of this diagram implies by the functoriality of $\otimes'$ the commutativity of the diagrams
\[
\xymatrix{
(F1\otimes'1')\otimes' Fx\ar[rr]^{(id\otimes'F_1)\otimes' id}\ar[d]_{r'_{F1}\otimes' id} && (F1\otimes' F1)\otimes' Fx\ar[d]^{F_\otimes(1,1)\otimes' id}
\\
F1\otimes' Fx && F(1\otimes 1)\otimes' Fx\ar[ll]^{Fd\otimes' id}
}
\]
for each object $x$, which is in turn equivalent to the commutativity of the bigger diagrams
\begin{equation}\label{varepsilon_cond3}
\xymatrix{
F1\otimes'(F1\otimes' Fx)\ar@/^1.5pc/[rrrd]^{a'_{F1,F1,Fx}}\ar@{}@<-4ex>[rr]|{\textcircled{N}} &&&
\\
F1\otimes'(1'\otimes' Fx)\ar[u]^{id\otimes'(F_1\otimes' id)}\ar@/_1pc/[dr]_{id\otimes'l'_{Fx}}\ar[r]^{a'_{F1,1',Fx}}
& 
(F1\otimes'1')\otimes' Fx\ar[rr]^{(id\otimes'F_1)\otimes' id}\ar[d]^{r'_{F1}\otimes' id} \ar@{}@<4ex>[l]|{\textcircled{C}}
&& 
(F1\otimes' F1)\otimes' Fx\ar[d]^{F_\otimes(1,1)\otimes' id}
\\
& F1\otimes' Fx && F(1\otimes 1)\otimes' Fx\ar[ll]^{Fd\otimes' id}
}
\end{equation}
obtained by adding diagrams $\textcircled{N}$ and $\textcircled{C}$, which commute by the naturality of $a'$ and the coherence axioms on $a',l',r'$. Let us consider now the diagram, whose arrows are all invertible:
\[
\xymatrix{
F(1\otimes(1\otimes x))\ar[dd]_{Fa_{1,1,x}}\ar[rd]^{F(id\otimes l_x)}  \ar@{}@<-9ex>[r]|{\textcircled{C}}
&& 
F1\otimes' F(1\otimes x)\ar[ll]_{F_\otimes(1,1\otimes x)}\ar[d]_{id\otimes'Fl_x} \ar@{}@<-4ex>[rr]|{\textcircled{A}}
\ar@{}@<4ex>[ll]|{\textcircled{N}}
&& 
F1\otimes'(F1\otimes' Fx)\ar[dd]^{a'_{F1,F1,Fx}}\ar[ll]_{id\otimes'F_\otimes(1,x)}
\\
& 
F(1\otimes x) \ar@{}@<5ex>|{\textcircled{N}}
& 
F1\otimes' Fx\ar[l]_{F_\otimes(1,x)}
& 
F1\otimes'(1'\otimes' Fx)\ar[l]_{id\otimes' l'_{Fx}}\ar[ru]^{id\otimes'(F_1\otimes' id)}\ar@{}@<-5ex>|{\textcircled{B}} 
& 
\\
F((1\otimes 1)\otimes x)\ar[ru]_{F(d\otimes id)} 
&& 
F(1\otimes 1)\otimes' Fx\ar[ll]^{F_\otimes(1\otimes 1,x)}\ar[u]_{Fd\otimes'id} 
&& 
(F1\otimes' F1)\otimes' Fx\ar[ll]^{F_\otimes(1,1)\otimes' id}
}
\]
In this diagram, the outer hexagon commutes because of axiom (SF1) on $(F,F_\otimes)$, subdiagrams labelled $\textcircled{N}$ by the naturality of $F_\otimes(x,y)$ and subdiagram $\textcircled{C}$ by the coherence of $a,l,r$. It follows that the commutativity of $\textcircled{A}$ is equivalent to the commutativity of $\textcircled{B}$, and $\textcircled{B}$ is the outer diagram in (\ref{varepsilon_cond3}), which is indeed commutative. Since $F1\otimes'-$ is an equivalence of categories, the diagram
\begin{equation}\label{varepsilon_cond4}
\xymatrix{
1'\otimes'Fx\ar[rr]^{F_1\otimes' id}\ar[d]_{l'_{Fx}} && F1\otimes' Fx\ar[d]^{F_\otimes(1,x)}
\\
Fx && F(1\otimes x)\ar[ll]^{F(l_x)}
}
\end{equation}
commutes. The commutativity of the first diagram in (SF3) follows similarly by applying the functor $Fx\otimes'-$ to the diagram (\ref{varepsilon_cond1'}). By the uniqueness of $F_1$, it only remains to check that (\ref{varepsilon_cond1'}), or equivalently, the product functor of (\ref{varepsilon_cond1'}) with $(F1)^*$ on the right, commutes when $F_1$ is given by (\ref{expressio_varepsilon}), and this follows from the diagram
\[ 
\xymatrix{
F1 \ar@{}@<-5ex>[r]|{\textcircled{N}}
& 
1'\otimes' F1\ar[l]-_{l'_{F1}} \ar[r]^-{\eta\otimes' id}\ar@{}@<-5ex>[r]|{\textcircled{F}}
& 
(F1^*\otimes' F1)\otimes' F1 \ar@{}@<-5ex>[r]|{\textcircled{N}}
& 
F1^*\otimes'(F1\otimes' F1)\ar[l]_{a'}\ar[r]^-{id\otimes'\varphi_{1,1}}\ar@{}@<-5ex>[r]|{\textcircled{A}}
& 
F1^*\otimes' F(1\otimes 1)\ar[d]|{id\otimes' Fd}
\\
1'\ar[u]^{F_1} 
& 
1'\otimes' 1'\ar[l]_-{d'}\ar[u]|{id\otimes'F_1}\ar[d]^{d'}\ar[r]_-{\eta\otimes' id}\ar@{}@<-5ex>[r]|{\textcircled{N}}
& 
(F1^*\otimes' F1)\otimes' 1'\ar[u]|{(id\otimes' id)\otimes'F_1}\ar[d]|{r'_{F1^*\otimes' F1}}\ar@{}@<-5ex>[rr]|{\textcircled{C}}
& 
F1^*\otimes'(F1\otimes' 1')\ar[l]^-{a'}\ar[r]_-{id\otimes' r'_{F1}}\ar[u]|{id\otimes'(id\otimes'F_1)} 
& 
F1^*\otimes' F1\ar@/^1.5pc/[lld]^-{id}
\\
& 
1'\ar[lu]^{id} \ar[r]_{\eta}
& 
F1^*\otimes' F1
&&
}
\] 
for any isomorphism $\eta:1'\to (F1)^*\otimes' F1$, where the subdiagrams labelled $\textcircled{N}$ are naturality squares, $\textcircled{F}$ commutes by the functoriality of $\otimes'$, $\textcircled{C}$ by the coherence axioms on $a',l',r'$. By hypothesis, the outer diagram commutes and hence, $\textcircled{A}$ also commutes. 

(2) Let us consider the diagram
\[
\xymatrix{
1'\otimes' Fx\ar[ddd]_{id\otimes'\tau_x}\ar[dr]^{F_1\otimes' id}\ar@/^1.5pc/[rrrd]^{l'_{Fx}} 
\ar@{}@<-13ex>[r]|{\textcircled{b}}
& & & 
\\
& 
F1\otimes' Fx\ar[r]^{F_\otimes(1,x)}\ar[d]_{\tau_1\otimes'\tau_x} \ar@{}@<-5ex>[r]|{\textcircled{c}}
\ar@{}@<5ex>[r]|{\textcircled{a}}
& 
F(1\otimes x)\ar[r]_-{Fl_x}\ar[d]_{\tau_{1\otimes x}} \ar@{}@<-5ex>[r]|{\textcircled{d}}
& 
Fx\ar[d]^{\tau_x}
\\
& 
F'1\otimes' F'x\ar[r]_{F'_\otimes(1,x)}  \ar@{}@<-5ex>[r]|{\textcircled{e}}
& 
F'(1\otimes x)\ar[r]^-{F'l_x}
& 
F'x
\\
1'\otimes'F'x\ar[ru]_{F'_1\otimes' id}\ar@/_1.5pc/[rrru]_{l'_{F'x}} 
&&&
}
\]
The outer diagram commutes by the naturality of $l'$, diagrams $\textcircled{a}$ and $\textcircled{e}$ by axiom (SF3) on $\FF$ and $\FF'$, respectively, diagram $\textcircled{c}$ by (T1) on $\tau$, and diagram $\textcircled{d}$ by the naturality of $\tau$. Since all involved morphisms are invertible, $\textcircled{b}$ also commutes, and this diagram is the same as the diagram
\[
\xymatrix{
1'\otimes' Fx\ar[rr]^{F_1\otimes'id}\ar[d]_{F_1'\otimes' id} & &F1\otimes' Fx\ar[d]^{\tau_1\otimes' id}
\\
F'1\otimes' Fx\ar@{:}[rr]\ar[dr]_{id\otimes'\tau_x} && F'1\otimes' Fx\ar[dl]^{id\otimes' \tau_x}
\\
& F'1\otimes' F'x &
}
\]  
whose commutativity is equivalent to axiom (T2) because $\tau$ is invertible and $-\otimes' Fx$ is an equivalence.
\end{proof}

\noindent
\bibliography{bibliografia}{}
\bibliographystyle{plain}

\end{document}